\documentclass[12pt,reqno,a4paper]{amsart}
\usepackage{amsmath,amssymb,amsthm,amsfonts}
\usepackage[dvips]{graphicx}
\usepackage{caption}
\usepackage{subcaption}
\usepackage{tikz}

\theoremstyle{plain}
\newtheorem{thm}{Theorem}[section]
\newtheorem{lem}[thm]{Lemma}

\theoremstyle{definition}
\newtheorem{defn}{Definition}[section]

\theoremstyle{remark}

\theoremstyle{plain}
\newtheorem{thmx}{Theorem}

\newcommand{\Rr}{\mathbb{R}}
\newcommand{\Nn}{\mathbb{N}}
\newcommand{\Zz}{\mathbb{Z}}

\newcommand{\I}{\mathcal{I}}
\newcommand{\ba}{\boldsymbol{a}}

\begin{document}

\title[Interval piecewise contractions]
{Hausdorff dimension of the exceptional set of interval piecewise affine contractions}

\author[Gaiv\~ao]{Jos\'e Pedro Gaiv\~ao}
\address{Departamento de Matem\'atica and CEMAPRE/REM, ISEG\\
Universidade de Lisboa\\
Rua do Quelhas 6, 1200-781 Lisboa, Portugal}
\email{jpgaivao@iseg.ulisboa.pt}

\date{\today}

\begin{abstract}
Let $I=[0,1)$,  $-1<\lambda<1$ and  $f\colon I\to I$ be a piecewise $\lambda$-affine map of the interval $I$, i.e.,  there exist a partition $0=a_0<a_1<\cdots< a_{k-1}<a_k=1$ of the interval $I$ into $k\geq2$ subintervals and $b_1,\ldots, b_k\in\Rr$ such that $f(x)=\lambda x+ b_i$ for every $x\in[a_{i-1},a_{i})$ and $i=1,\ldots,k$.  The exceptional set $\mathcal{E}_f$ of $f$ is the set of parameters $\delta\in\Rr$ such that $R_\delta\circ f$ is not asymptotically periodic, where $R_\delta\colon I\to I$ is the rotation of angle $\delta$.  In this paper we prove that $\mathcal{E}_f$ has zero Hausdorff dimension.  We derive this result from a more general theorem concerning piecewise Lipschitz contractions on $\Rr$ that has independent interest.
\end{abstract}

%
\maketitle

\section{Introduction}

Let $I=[0,1)$ and $-1<\lambda<1$.  We say that an interval map $f:I\to I$ is \textit{piecewise $\lambda$-affine} if there exist $k\geq 2$,  real numbers $b_1,\ldots, b_k$ and a partition of the interval $I$,
$$
0=a_0<a_1<\cdots<a_{k-1}<a_k=1,
$$
such that $f(x)=\lambda x+ b_i$ for every $x\in[a_{i-1}, a_i)$ and $i=1,\ldots, k$.  Given a piecewise $\lambda$-affine map $f$, consider the one-parameter family of piecewise $\lambda$-affine maps $f_\delta\colon I\to I$ defined by 
$$
f_\delta= R_\delta\circ f,
$$
where $R_\delta(x)=\{x+\delta\}$ is the rotation of angle $\delta\in\Rr$ and $\{\cdot\}$ denotes the fractional part.  See Figure~\ref{fig1} for an illustration of the graph of $f_\delta$. 

We are interested in the dimension of the exceptional set
$$
\mathcal{E}_f=\{\delta\in\Rr\colon f_\delta\text{ is not asymptotically periodic}\}.
$$
A map $f\colon I\to I$ is \textit{asymptotically periodic} if $f$ has at most a finite number of periodic orbits and the $\omega$-limit set $\omega(f,x)$ of any $x\in I$ is a periodic orbit of $f$.  We recall that $\omega(f,x)$ is the set of accumulation points of the forward orbit of $x$ under $f$.  It is known that $\mathcal{E}_f$ has zero Lebesgue measure \cite[Theorem 1.1]{NPR18}, but the question of the Hausdorff dimension of $\mathcal{E}_f$ was still open.  

\newpage
In this paper, we settle this question.

\begin{thmx}\label{thmA}
$$\dim_ H \mathcal{E}_f=0.$$
\end{thmx}

A notable example of piecewise $\lambda$-affine maps is the family of contracted rotations, $f(x)=\{\lambda x+b\}$ with $0<\lambda<1$ and $1-\lambda<b<1$. Contracted rotations have been extensively studied by many authors either as a dynamical system or related to applications, e.g.  \cite{B93, BC99, LN18, JO19, BKLN20,BS20}. In the case of contracted rotations, the exceptional set $\mathcal{E}_f$ has a Cantor structure and in this case Theorem~\ref{thmA} was proved by Laurent and Nogueira in \cite[Theorem 5]{LN18} by exploiting a combinatorial structure, associated to $\mathcal{E}_f$, which is reminiscent of the classical Stern-Brocot tree.  Janson and \"Oberg improved in \cite{JO19} the result of Laurent and Nogueira by considering other gauge functions in the definition of the Hausdorff dimension. 

\begin{figure}
\centering
\begin{subfigure}{.3\textwidth}
  \centering
  \includegraphics[width=0.9\linewidth]{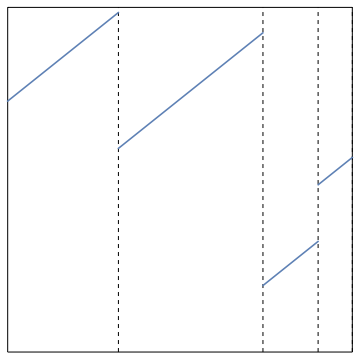}
  \caption{$f_0$}
\end{subfigure}%
\begin{subfigure}{.3\textwidth}
  \centering
  \includegraphics[width=0.9\linewidth]{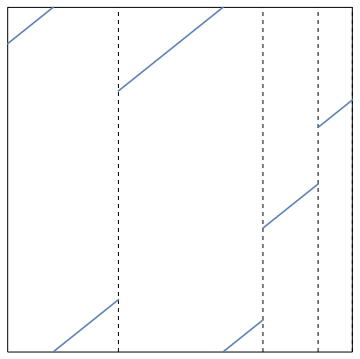}
  \caption{$f_{\frac16}$}
\end{subfigure}
\begin{subfigure}{.3\textwidth}
  \centering
  \includegraphics[width=0.9\linewidth]{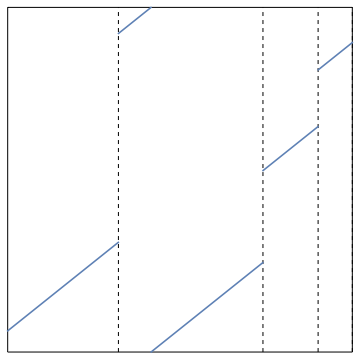}
  \caption{$f_{\frac13}$}
\end{subfigure}
\caption{Plots of $f_\delta$. }
\label{fig1}
\end{figure}

For general piecewise $\lambda$-affine maps, Theorem~\ref{thmA} was proved in \cite{B20} by Pires using the theory of $b$-adic expansions and under the assumptions that $f$ is injective, $\lambda^{-1}=b$ is a positive integer $\geq k$ and the connected components of $I\setminus f(I)$ have equal length. In order to remove all these assumptions and prove Theorem~\ref{thmA} for any piecewise $\lambda$-affine map $f$ we use a different strategy inspired by a recent work dealing with piecewise increasing contractions \cite{GN22}.

We will deduce Theorem~\ref{thmA} from  a more general result, Theorem~\ref{thmB} below.   We say that a function $f\colon \Rr\to\Rr$ is a \textit{piecewise contraction} if it has a finite number of discontinuity points and on each connected component $D$ of the domain of continuity of $f$ the restriction map $f|_D$ is a Lipschitz contraction.
Let $k\geq2$ and
$$
\Omega_{k} = \{(a_1,\ldots, a_{k-1})\in\Rr^{k-1}\colon a_1<\cdots<a_{k-1}\}.
$$
A finite collection of Lipschitz contractions $\Phi=\{\phi_i\colon \Rr\to \Rr\}_{i=1}^k$ is called an iterative function system (IFS).  An IFS $\Phi=\{\phi_1,\ldots,\phi_k\}$ together with $\ba\in\Omega_{k}$, determine a  piecewise contraction $f=f_{\Phi,\ba}\colon \Rr\to \Rr$ defined by 
\begin{equation}\label{def of f}
f(x)=\begin{cases}\phi_1(x),& x\in(-\infty,a_1),\\
\phi_i(x),& x\in[a_{i-1},a_{i}) \quad i\in\{2,\ldots,k-1\},\\
\phi_k(x),& x\in [a_{k-1},+\infty).
\end{cases}
\end{equation}
Notice that $(\Phi,\ba)\mapsto f_{\Phi,\ba}$ is not injective, i.e., a piecewise contraction is not uniquely determined by a pair $(\Phi,\ba)$.
In \cite{NPR18} the authors prove the following result.

\begin{thm}[{\cite[Theorem 1.4]{NPR18}}]\label{th:Arnaldo}
Let $\Phi=\{\phi_1,\ldots,\phi_k\}$  be an IFS.  There is a Lebesgue full measure set $W\subset \Rr$ such that 
for every $\ba \in\Omega_{k} \cap W^{k-1}$, the map $f_{\Phi,\ba}$ is asymptotically periodic and has at most $k$ periodic orbits.
\end{thm}

We consider a very specific perturbation of $f_{\Phi,\ba}$.
To simplify the notation, we shall write $\ba+\delta = (a_1+\delta,\ldots,a_{k-1}+\delta)$. Notice that $\ba+\delta\in\Omega_k$ for every $\delta\in\Rr$.

We say that an IFS  is \textit{injective} if all its contractions are injective functions. 
Under the assumption that the IFS is injective, we prove the following result.

\begin{thmx}\label{thmB}
Let $\Phi=\{\phi_1,\ldots,\phi_k\}$  be an injective IFS and $\ba\in\Omega_k$. Then
$$
\dim_H\left\{ \delta\in  \Rr \colon f_{\Phi,\ba+\delta}\text{ is not asymptotically periodic}\right\}=0.
$$
\end{thmx}

The injectivity assumption of Theorem~\ref{thmB} can be weakened to the assumption that the functions of the IFS have a finite number of local extrema.  This hypothesis guarantees that the pre-image of any point is a finite set, a crucial property that we use in our arguments to prove Theorem~\ref{thmB}.  By Sard's theorem for Lipschitz functions, the pre-image of almost every point is a finite set. However, this property is not sufficient for our arguments to work in the general situation.  The injectivity of the IFS is a natural condition,  and piecewise contractions appearing in applications satisfy this condition.  Moreover,  assuming the injectivity of the IFS,  the authors in \cite{CCG21} have obtained a spectral decomposition of the attractor of a piecewise contraction,  i.e.,  the attractor is a finite union of periodic orbits together with a finite union of Cantor sets.  Our Theorem~\ref{thmB} excludes the Cantor attractors in a very strong metric sense,  i.e.,  piecewise contractions can only have Cantor attractors for a parameter set of zero Hausdorff dimension. 

As a final remark,  the definition \eqref{def of f} of $f$ at the points $\{a_1,\ldots, a_{k-1}\}$ is not relevant for proving Theorem~\ref{thmB},  i.e.,  the proof of Theorem~\ref{thmB} can be adapted to any other choice of values $f(a_i)\in~\left\{f(a_i^{-}),f(a_i^{+})\right\}$. 

The rest of the paper is organized as follows.  In Section~\ref{prelim} we collect the lemmas that we need to prove Theorem~\ref{thmB} and complete its proof in Section~\ref{proof thmB}.  Then,  in Section~\ref{proof thmA}, we use Theorem~\ref{thmB} to deduce Theorem~\ref{thmA}.

\section{Preliminary results}\label{prelim}

Throughout this section, let $\Phi=\{\phi_1,\ldots,\phi_k\}$  be an injective IFS and $\ba=(a_1,\ldots,a_{k-1})\in\Omega_k$ with $k\geq2$.  Define
$$
\lambda_{\Phi}:=\max_i\, \text{Lip}(\phi_i)\quad\text{and}\quad
r_\Phi :=\frac{1+\lambda_{\Phi}}{1-\lambda_\Phi} \max_i\, |z_i|,
$$
where $\text{Lip}(\phi_i)$ denotes the Lipschitz constant of $\phi_i$ and $z_i$ is the unique fixed point of $\phi_i$. Clearly, $0\leq \lambda_\Phi<1$. 

Our first observation is that the recurrent dynamics of $f_{\Phi,\ba}$ occurs inside an attracting compact interval 
$$
K_{\Phi}:=[-2r_{\Phi},2r_{\Phi}].
$$
Given $n\in\Nn$ and $\omega\in\{1,\ldots,k\}^n$, let $$
\phi_\omega:=\phi_{\omega_n}\circ\cdots\circ\phi_{\omega_1}.
$$  

\begin{lem}\label{lem:compact} The following holds:
\begin{enumerate}
\item $\phi_\omega(K_{\Phi})\subset K_{\Phi}$ for every $\omega\in\{1,\ldots,k\}^n$ and $n\in\Nn$, 
\item $f_{\Phi,\ba}(K_{\Phi})\subset K_{\Phi}$,
\item For every $x\in \Rr$ there is $n\geq0$ such that $f^n_{\Phi,\ba}(x)\in K_{\Phi}$,
\item $\omega(f_{\Phi,\ba},x)\subset K_{\Phi}$ for every $x\in\Rr$.
\end{enumerate}
\end{lem}

\begin{proof}
To show (1) let $|x|\leq 2 r_{\Phi}$.  Then
\begin{align*}
|\phi_i(x)| = |\phi_i(x)-z_i+z_i| &\leq \lambda_{\Phi}| x-z_i|+|z_i| \\
&\leq 2\lambda_{\Phi}r_{\Phi}+ |z_i|(1+\lambda_{\Phi})\\
&\leq 2\lambda_{\Phi}r_{\Phi}+ \max_i|z_i|(1+\lambda_{\Phi})\\
&=2\lambda_{\Phi}r_{\Phi} + (1-\lambda_{\Phi})r_{\Phi}\\
&=(1+\lambda_{\Phi})r_{\Phi}\\
&<2r_{\Phi},
\end{align*}
which proves (1).  Item (2) follows immediately from (1).  
Next,  to prove (3),  let $|x|\leq r_{\Phi}$. Repeating the above estimates we get $|\phi_i(x)|\leq r_{\Phi}$.
For any $\omega\in\{1,\ldots,k\}^n$ and $n\in\Nn$,  this shows that $|\phi_{\omega}(x)|\leq r_{\Phi}$ whenever $|x|\leq r_{\Phi}$.  Consequently,  the unique fixed point $z_\omega$ of $\phi_\omega$ satisfies $|z_{\omega}|\leq r_{\Phi}$.  
Now,  given any $x\in\Rr$, choose $n\geq1$ such that $\lambda_{\Phi}^n<\frac{r_{\Phi}}{|x|+r_{\Phi}}$.  Then, for any $\omega\in\{1,\ldots,k\}^n$
\begin{align*}
|\phi_{\omega}(x)| &= |\phi_{\omega}(x) - z_\omega + z_\omega|\\
&\leq \lambda_{\Phi}^{n} |x-z_\omega|+ |z_\omega|\\
&\leq \lambda_{\Phi}^{n}(|x|+r_{\Phi})+r_{\Phi}\\
&<2r_{\Phi}.
\end{align*}
This proves (3).  Finally,  (4) follows immediately from (2) and (3). 
\end{proof}
 
Let
$$
S_{\ba} := \{a_1,\ldots,a_{k-1}\}.
$$
\begin{defn}
We say that the pair $(\Phi,\ba)$ has a \textit{singular connection} if there exist $n\in\Nn$ and $\omega\in\{1,\ldots,k\}^n$ such that
$$\phi_\omega(S_{\ba})\cap S_{\ba}\neq \emptyset.$$
\end{defn}

\begin{lem}\label{lem:sing}
The set
$$ \{\delta\in  \Rr \colon (\Phi,\ba+\delta)\text{ has a singular connection}\}\
$$
is countable.
\end{lem}

\begin{proof}
Given $\omega\in\{1,\ldots,k\}^n$, the map $\phi_\omega$ is a Lipschitz contraction on $\Rr$.  Given $(i,j)\in\{1,\ldots,k-1\}^2$, the map $\Rr\ni\delta\mapsto \phi_\omega(a_i+\delta)-a_j$ is also a Lipschitz contraction on $\Rr$,  hence it has a unique fixed point, say $z_{\omega,i,j}\in\Rr$.   Let
$$
\Delta:=\bigcup_{n\in\Nn}\bigcup_{\omega\in\{1,\ldots,k\}^n}\bigcup_{i=1}^{k-1}\bigcup_{j=1}^{k-1} \{z_{\omega,i,j}\}.
$$
Observe that $(\Phi,\ba+\delta)$ has a singular connection if and only if $\delta \in\Delta$.  Since $\Delta$ is a countable set,  the claim follows.
\end{proof}

\begin{defn}
   Given  $x\in K_{\Phi}$ and $n\in\Nn$,  we say that $x$ is an $n$-\textit{regular point of $(\Phi,\ba)$} if $f_{\Phi,\ba}^j(x)\notin S_{\ba}$ for every $0\leq j<n$.  
\end{defn}

Let $D^{(0)}_{\ba} := S_{\ba}$,  $D^{(n)}_{\ba} :=f_{\Phi,\ba}^{-1}(D^{(n-1)}_{\ba})$ for $n\geq1$ and
$$
Q^{(n)}_{\ba}:=\bigcup_{i=0}^{n-1} D^{(i)}_{\ba},\quad n\in\Nn.
$$
Notice that the sets $\{Q^{(n)}_{\ba}\}_{n\geq0}$ are finite. Indeed, this follows from the fact that $S_{\ba}$ is finite and $\Phi$ is injective.  
Given $x\in K_{\Phi}$,  it is also clear that $x$ is an $n$-regular point of $(\Phi,\ba)$ if and only if $x\notin Q^{(n)}_{\ba}$.

Let $X_1:=(-\infty,a_1)$, $X_i:=(a_{i-1},a_{i})$ with $i=2,\ldots, k-1$, and $X_k:=(a_{k-1},+\infty)$.  By construction,  these  open intervals are disjoint and their union equals $\Rr\setminus S_{\ba}$.

\begin{defn}
Given $n\in\Nn$, a tuple $(\omega_0,\omega_1,\ldots,\omega_{n-1})\in\{1,\ldots,k\}^{n}$ is called an \textit{itinerary of order $n$ of $(\Phi,\ba)$} if there is an $n$-regular point $x$ of $(\Phi,\ba)$ such that $f^j_{\Phi,\ba}(x)\in X_{\omega_j}$ for every $0\leq j<n$.   
\end{defn}

We define the set of all itineraries of order $n$ of $(\Phi,\ba)$,  
$$
\I_{\Phi,\ba}^{(n)}:=\left\{\omega\in\{1,\ldots,k\}^n\colon \omega \text{ is an itinerary of order $n$ of $(\Phi,\ba)$}\right\}.
$$
The set $\I_{\Phi,\ba}^{(n)}$ is in a one-to-one correspondence with the set of  connected components of $K_{\Phi}\setminus Q^{(n)}_{\ba}$. Indeed, for each connected component $J$ of $K_{\Phi}\setminus Q^{(n)}_{\ba}$, all points in $J$ are $n$-regular points of $(\Phi,\ba)$ and because the  sets $J, f_{\Phi,\ba}(J), \ldots, f^{n-1}_{\Phi,\ba}(J)$ are subintervals of $K_{\Phi}$ not intersecting $S_{\ba}$, we conclude that there is an itinerary of order $n$ of $(\Phi,\ba)$, say $(\omega_0,\omega_1,\ldots,\omega_{n-1})\in\I_{\Phi,\ba}^{(n)}$,  such that $f^{j}_{\Phi,\ba}(J)\subset X_{\omega_j}$ for every $0\leq j<n$.

Given $\varepsilon\geq 0$, we enlarge the set $\I_{\Phi,\ba}^{(n)}$ as follows,
$$
\I^{(n)}_{\Phi,\ba}(\varepsilon):=\bigcup_{|\delta|\leq \varepsilon}\I^{(n)}_{\Phi,\ba+\delta}.
$$
The following result establishes that the number of itineraries grows subexponentially,  a crucial property to prove Theorem~\ref{thmB}.

\begin{lem}\label{lem:entropy}Suppose that $(\Phi,\ba)$ has no singular connections.  Then 
$$
\lim_{\varepsilon\to0^+}\limsup_{n\to+\infty}\frac{1}{n}\log\#\I^{(n)}_{\Phi,\ba}(\varepsilon) = 0.
$$
\end{lem}

\begin{proof}
Given $\rho>1$, let $m=\lceil \log 2/\log \rho\rceil$ and $\tau(m,\ba)>0$ be the minimum distance between any pair of distinct points of $Q^{(m)}_{\ba}$. 
Notice that the sets $\{D^{(n)}_{\ba}\}_{n\geq0}$ are pairwise disjoint. Indeed, suppose that there is $x\in D^{(n_1)}_{\ba}\cap D^{(n_2)}_{\ba}$ with $n_2>n_1\geq0$. Then $f_{\Phi,\ba}^{n_1}(x)=a_i$ and $f_{\Phi,\ba}^{n_2}(x)=a_j$ for some $i,j\in\{1,\ldots,k-1\}$. This implies that $f_{\Phi,\ba}^{n_2-n_1}(a_i)=a_j$,  contradicting the assumption that $(\Phi,\ba)$ has no singular connections.  Since
$
Q^{(m)}_{\ba}=D^{(0)}_{\ba}\cup D^{(1)}_{\ba}\cup\cdots\cup D^{(m-1)}_{\ba}
$
and $\Phi$ is injective, there is an $\varepsilon_0=\varepsilon_0(m)>0$ such that the set-valued map $(-\varepsilon_0,\varepsilon_0)\ni \delta \mapsto Q^{(m)}_{\ba+\delta}$ varies continuously in the Hausdorff metric of compact subsets of $\Rr$, and for every $|\delta|<\varepsilon_0$,  the set $Q^{(m)}_{\ba+\delta}$ has the same number of elements of $Q^{(m)}_{\ba}$ and
$$
\tau(m):=\inf_{ |\delta|<\varepsilon_0}\tau(m,\ba+\delta)>0.
$$
Let $\alpha_n(\varepsilon) := \# \I^{(n)}_{\Phi,\ba}(\varepsilon)$ with $0\leq \varepsilon<\varepsilon_0$.  Clearly,  $\alpha_m(\varepsilon)=\alpha_m(0)=\#\I_{\Phi,\ba}^{(m)}$.  
Now,  choose $n_0\geq0$ sufficiently large so that for every $|\delta|<\varepsilon_0$, every $n\geq n_0$  and every connected component $J$ of $K_{\Phi}\setminus Q^{(n)}_{\ba+\delta}$,  the length of the interval  $f^n_{\Phi,\ba+\delta}(J)$ is smaller than $\tau(m)$.  Notice that this is possible since $|f^n_{\Phi,\ba+\delta}(J)|\leq \lambda_{\Phi}^n| J|\leq 4\lambda_{\Phi}^n r_{\Phi} $.  
Since, for every $|\delta|<\varepsilon_0$, any interval $J\subset K_{\Phi}$ with length $<\tau(m)$ will intersect $Q^{(m)}_{\ba+\delta}$ in at most a single point, we conclude that 
$$
\alpha_{n+m}(\varepsilon)\leq 2\alpha_{n}(\varepsilon),\quad \forall\,n\geq n_0.
$$ 
Thus, $\alpha_{n }(\varepsilon)\leq 2^{\frac{n-n_0}{m}} \alpha_{n_0}(\varepsilon)$ for every $n\geq n_0$. By our choice of $m$, we get $\alpha_{n }(\varepsilon)\leq C \rho^n$ for every $n\geq n_0$ where $C:= 2^{-n_0/m}\alpha_{n_0}(\varepsilon_0)$. This shows that 
$$
\limsup_{n\to\infty}\frac{\alpha_{n }(\varepsilon)}{n}\leq \log\rho,\quad \forall\, \varepsilon\in[0,\varepsilon_0).
$$
As $\rho>1$ is arbitrary,  the claim follows.
\end{proof}

Let 
$$
Q_{\ba}:=\left(\bigcup_{n\geq 0}Q^{(n)}_{\ba}\right)\cap K_{\Phi}.
$$

The following result is adapted from the results of \cite{NPR18} (cf. \cite[Theorem 20]{GN22}).  We include here a proof for the convenience of the reader. 

\begin{lem}\label{lem:finite}
Suppose that $(\Phi,\ba)$ has no singular connections. If $Q_{\ba}$ is finite, then $f_{\Phi,\ba}$ is asymptotically periodic. 
\end{lem}

\begin{proof}
To simplify the notation, let $f=f_{\Phi,\ba}$. 
Let $\mathcal{P}=\{J_\ell\}_{\ell=1}^m$ denote the collection of connected components of $K_{\Phi}\setminus Q_{\ba}$.  This collection is finite by hypothesis.  Notice that there is a map $\sigma\colon \{1,\ldots,m\}\to\{1,\ldots,m\}$ such that $f(J_\ell)\subset J_{\sigma(\ell)}$ for every $\ell\in\{1,\ldots,m\}$. Indeed,  suppose by contradiction that there is $J_\ell\in\mathcal{P}$ such that $f(J_\ell)\cap Q_{\ba}\neq\emptyset$.   Then $J_\ell\cap f^{-1}
(Q_{\ba})\neq \emptyset$.  But $f^{-1}(Q_{\ba})\cap K_{\Phi}\subset Q_{\ba}$, which implies that $J_\ell\cap Q_{\ba}\neq\emptyset$, thus a contradiction. 

Now we show that $\omega(f,x)$ is a periodic orbit for every $x\in \Rr$.  By Lemma~\ref{lem:compact},   we may assume that $x\in K_{\Phi}$.  We split the proof in two cases:
\begin{enumerate}
\item When $x\in K_{\Phi}\setminus Q_{\ba}$,  then $x\in J_{\ell_0}$ for some $\ell_0\in\{1,\ldots,m\}$.  Thus,  $f^n(x)\in J_{\ell_n}$ for every $n\geq0$ where $(\ell_n)_{n\geq0}$ is the sequence in $\{1,\ldots,m\}$ defined by $\ell_{n+1}=\sigma(\ell_n)$ for every $n\geq0$.  Clearly,  the sequence $(\ell_n)_{n\geq0}$ is eventually periodic,  i.e., there must exist $q\geq0$ and $p\geq 1$ such that $\ell_{n+p}=\ell_n$ for every $n\geq q$.  We assume that $p$ is the smallest positive integer with that property, i.e.,  $p$ is the period.  In particular,  we have  $f^p(J_{\ell_{q}})\subset J_{\ell_{q}}$. Let  $\omega\in\{1,\ldots, k\}^p$ such that $f^p|_{J_{\ell_{q}}} = \phi_\omega|_{J_{\ell_{q}}} $.  Then,  $f^{np+q}(x)\to z_\omega$ as $n\to\infty$ where  $z_{\omega}$ is the unique fixed point of $\phi_\omega$ which belongs to the closure of the interval $J_{\ell_{q}}$.  By hypothesis,  $(\Phi,\ba)$ has no singular connections, which implies that $z_\omega \in J_{\ell_q}$,  and thus $z_\omega$ is a periodic point of $f$ of period $p$.  Therefore, $\omega(f,x)=\{z_\omega,f(z_\omega),\ldots,f^{p-1}(z_\omega)\}$.
\item When $x\in Q_{\ba}$,  then two situations can happen.  Either the forward orbit of $x$ under $f$ belongs to $Q_{\ba}$ and thus it is periodic or the forward orbit of $x$ under $f$ eventually leaves $Q_{\ba}$.  The former situation cannot happen,  because $(\Phi,\ba)$ has no singular connections.  Thus,  there exists $n_0\geq1$ such that $y:=f^{n_0}(x)\in K_{\Phi}\setminus Q_{\ba}$.  But then $\omega(f,x)=\omega(f,y)$ and we know from the previous case that $\omega(f,y)$ is a periodic orbit. 
\end{enumerate}
At last,  because $\mathcal{P}$ is finite,  $f$ has only a finite number of periodic orbits and the claim follows. 
\end{proof}

Given $\varepsilon\geq 0$, let 
$$
\Omega^\varepsilon_{\Phi,\ba}:=\bigcap_{m\geq1}\overline{\bigcup_{n\geq m}\bigcup_{\omega\in \I^{(n)}_{\Phi,\ba}(\varepsilon)}\{\phi_\omega(0)\}}.
$$
Notice that $\Omega^\varepsilon_{\Phi,\ba}\subset K_{\Phi}$ and $\Omega^\varepsilon_{\Phi,\ba}$ is compact (see Lemma~\ref{lem:compact}). 
\begin{lem}\label{lem:Omega cover}
For every $n\in\Nn$, the set $\Omega^\varepsilon_{\Phi,\ba}$ can be covered by finitely many intervals of length $2(1+2r_\Phi)\lambda_\Phi^n$ centered at the points $\phi_\omega(0)$ with $\omega\in \I^{(n)}_{\Phi,\ba}(\varepsilon)$.
\end{lem}

\begin{proof}
Let $x\in \Omega^\varepsilon_{\Phi,\ba}$ and $n\in\Nn$.  By definition of $\Omega^\varepsilon_{\Phi,\ba}$, there is an increasing sequence of positive integers $m_j\to\infty$ and a sequence $\omega^{(j)}\in \I^{(m_j)}_{\Phi,\ba}(\varepsilon)$ such that $\phi_{\omega^{(j)}}(0)\to x$ as $j\to\infty$. Let $j\geq1$ be sufficiently large such that $m_j\geq n$ and $|x-\phi_{\omega^{(j)}}(0)|\leq \lambda_\Phi^n$. Denote by $\omega^{(j,n)}=(\omega^{(j)}_{m_j-n+1},\ldots,\omega^{(j)}_{m_j})\in \I^{(n)}_{\Phi,\ba}(\varepsilon) $ the last $n$ entries of $\omega_j$. Then
$$
|\phi_{\omega^{(j)}}(0)-\phi_{\omega^{(j,n)}}(0)| =|\phi_{\omega^{(j,n)}}(y)-\phi_{\omega^{(j,n)}}(0)|
\leq \lambda_\Phi^n |y| 
\leq 2r_\Phi \lambda_\Phi^n,
$$
where $y:=\phi_{\omega^{(j)}_{m_j-n}}\circ\cdots\circ\phi_{\omega^{(j)}_{1}}(0)$ and,  by Lemma~\ref{lem:compact},  $|y|\leq 2r_{\Phi}$. Thus,
\begin{align*}
|x-\phi_{\omega^{(j,n)}}(0)|&\leq |x-\phi_{\omega^{(j)}}(0)|+|\phi_{\omega^{(j)}}(0)-\phi_{\omega^{(j,n)}}(0)|\\
& \leq (1+2r_\Phi)\lambda_\Phi^n. 
\end{align*}
\end{proof}

\begin{lem}\label{lem:Omega}
Let $|\delta|<\varepsilon$. If $\Omega^\varepsilon_{\Phi,\ba}\cap S_{\ba+\delta}=\emptyset$, then $Q_{\ba+\delta}$ is finite.
\end{lem}
\begin{proof}
By hypothesis, there is $\tau>0$ such that
\begin{equation}\label{eq:hyp separation}
\min_{1\leq i<k}| a_i+\delta -\phi_{\omega}(0)|\geq \tau,\quad \forall\,\omega\in\I_{\Phi,\ba}^{(n)}(\varepsilon),\, n\in\Nn.
\end{equation}
Suppose, by contradiction,  that $Q_{\ba+\delta}$ is not finite.  Then,  
$$
Q_{\ba+\delta}^{(1)}\cap K_{\Phi}\varsubsetneq Q_{\ba+\delta}^{(2)}\cap K_{\Phi}\ \varsubsetneq Q_{\ba+\delta}^{(3)} \cap K_{\Phi}\varsubsetneq  \cdots
$$
So,  we can pick a sequence $(x_n)_{n\geq1}$ in $Q_{\ba+\delta}$ having the property that $x_n\in Q_{\ba+\delta}^{(n+1)}\cap K_{\Phi}\setminus Q_{\ba+\delta}^{(n)}$ for every $n\in\Nn$.  Thus,  $x_n$ is an $n$-regular point of $(\Phi,\ba+\delta)$ and $f^{n}_{\Phi,\ba+\delta}(x_n)=a_{j_n}+\delta$ for some $j_n\in \{1,\ldots,k-1\}$.  Let $\omega^{(n)}\in \I_{\Phi,\ba}^{(n)}(\varepsilon)$ denote the itinerary of order $n$ of $(\Phi,\ba+\delta)$ associated to $x_n$.  

Now, choose $n\in\Nn$ sufficiently large so that $2r_{\Phi}\lambda_{\Phi}^n<\tau$.  
Then, taking into account that $x_n\in K_{\Phi}$, we have
\begin{align*}
|a_{j_n}+\delta - \phi_{\omega^{(n)}}(0)|  &= |f^n_{\Phi,\ba+\delta}(x_n) - \phi_{\omega^{(n)}}(0)| \\
&=  |\phi_{\omega^{(n)}}(x_n) - \phi_{\omega^{(n)}}(0)|\\
&\leq \lambda_{\Phi}^n |x_n|\\
&\leq 2r_{\Phi}\lambda_{\Phi}^n\\
&<\tau,
\end{align*}
contradicting \eqref{eq:hyp separation}. 
\end{proof}

\section{Proof of Theorem~\ref{thmB}}\label{proof thmB}

We are now ready to prove Theorem~\ref{thmB}.  Let $\Phi=\{\phi_1,\ldots,\phi_k\}$  be an injective IFS and $\ba=(a_1,\ldots,a_{k-1})\in\Omega_k$ with $k\geq2$.   By Lemma~\ref{lem:sing}, the set 
$$
E=\{\delta\in\Rr\colon (\Phi,\ba+\delta)\text{ has a singular connection}\}
$$ 
is countable. Thus, it is sufficient to show that 
$$
Z:=\{\delta\in \Rr\setminus E\colon f_{\Phi,\ba+\delta}\text{ is not asymptotically periodic}\} 
$$
has zero Hausdorff dimension. By Lemma~\ref{lem:finite}, $Z\subset Z'$ where
$$
Z':=\{\delta\in \Rr\setminus E\colon Q_{\Phi,\ba+\delta}\text{ is not finite}\}.
$$
Given $\delta\in \Rr$ and $\varepsilon>0$, let $\Delta_{\varepsilon}(\delta):=(\delta-\varepsilon,\delta+\varepsilon)$ and $Z'_{\varepsilon}(\delta) := Z'\cap \Delta_{\varepsilon}(\delta)$. We claim that for every $d>0$ and $\delta\in Z'$ there is an $\varepsilon=\varepsilon(\delta,d)>0$ such that
$$
\mathcal{H}^d( Z'_{\varepsilon}(\delta))  = 0,
$$
where $\mathcal{H}^d$ is the $d$-dimensional Hausdorff measure.  This claim is sufficient to conclude the proof of Theorem~\ref{thmB}, since by Lindel\"of Lemma, $Z'$ is a countable union of sets $Z'_{\varepsilon_i}(\delta_i)$, each having zero $d$-dimensional Hausdorff measure. Hence, $\mathcal{H}^d(Z')=0$ for every $d>0$. This implies that $\dim_H Z' = 0$.

Now we prove the claim. Let $d>0$, $\delta_0\in Z'$ and $\varepsilon>0$ that will be chosen later in the proof. By Lemma~\ref{lem:Omega}, 
$$
Z'_{\varepsilon}(\delta_0)\subset \{\delta\in\Delta_{\varepsilon}(\delta_0)\colon \Omega_{\Phi,\ba_0}^\varepsilon \cap S_{\ba+\delta}\neq\emptyset\}
$$
where $\ba_0:=\ba+\delta_0$. According to Lemma~\ref{lem:Omega cover},  for each $n\in\Nn$, the set $\Omega_{\Phi,\ba_0}^\varepsilon$ can be covered by $\# \I^{(n)}_{\Phi,\ba_0}(\varepsilon)$ intervals of length $\ell_n:=2(1+2r_\Phi)\lambda_\Phi^n$ centered at the points $\phi_\omega(0)$ with $\omega\in \I^{(n)}_{\Phi,\ba_0}(\varepsilon)$. Thus, for each $n\in\Nn$, we can also cover $Z'_{\varepsilon}(\delta_0)$ using finitely many intervals of length $\ell_n$,
$$
Z'_{\varepsilon}(\delta_0)\subset \bigcup_{\omega\in \I^{(n)}_{\Phi,\ba_0}(\varepsilon)}\bigcup_{i=1}^{k-1}W_{\omega,i}
$$
where
$W_{\omega,i} = \left[y_{\omega,i}-\frac{\ell_n}{2},y_{\omega,i}+\frac{\ell_n}{2}\right]$, $n=|\omega|$ and
$y_{\omega,i}=\phi_\omega(0)-a_i$ .
Using this cover, it is easy to see that there is an $\varepsilon>0$ such that the $d$-dimensional Hausdorff measure of $Z'_{\varepsilon}(\delta_0)$ is zero. Indeed, for every $n\in\Nn$, 
\begin{align*}
\mathcal{H}_{2\ell_n}^d(Z'_\varepsilon(\delta_0))&=\inf\left\{\sum_{i=1}^\infty (\text{diam}\, U_i)^d\colon \bigcup_{i=1}^\infty U_i \supset Z'_\varepsilon(\delta_0),\, \text{diam}\, U_i < 2\ell_n\right\}\\
&\leq \sum_{\omega\in \I^{(n)}_{\Phi,\ba_0}(\varepsilon)}\sum_{i=1}^{k-1}(\text{diam}\,W_{\omega,i})^d \\&= \sum_{\omega\in \I^{(n)}_{\Phi,\ba_0}(\varepsilon)}\sum_{i=1}^{k-1}\ell_n^d\\
&= (k-1) 2^d(1+2r_\Phi)^d (\# \I^{(n)}_{\Phi,\ba_0}(\varepsilon))\lambda_\Phi^{nd}.
\end{align*}
Notice that, $(\Phi,\ba_0)$ has no singular connections because $\delta_0\in Z'$. Hence, by Lemma~\ref{lem:entropy}, there is an $\varepsilon=\varepsilon(\delta_0,d)>0$ such that 
$$
\lim_{n\to\infty}  (\# \I^{(n)}_{\Phi,\ba_0}(\varepsilon))\lambda_\Phi^{nd} = 0,
$$
from which it follows that $\mathcal{H}^d(Z'_{\varepsilon}(\delta_0))=\lim_{n\to\infty} \mathcal{H}_{2\ell_n}^d(Z'_\varepsilon(\delta_0))=0$, thus proving the claim.  \qed

\section{Proof of Theorem~\ref{thmA}}\label{proof thmA}

Recall that $I=[0,1)$ and let $f\colon I\to I$ be a piecewise $\lambda$-affine map as defined in the introduction with $-1<\lambda<1$.  We may suppose that $\lambda\neq0$,  otherwise $\mathcal{E}_f=\emptyset$ and the result trivially holds.  Let $f_\delta=R_\delta\circ f$ where $R_\delta(x)=\{x+\delta\}$ is the rotation map of angle $\delta\in\Rr$. Notice that $f_\delta$ is also a piecewise $\lambda$-affine map and
$$
 \mathcal{E}_{f}=\mathcal{E}_{f_\delta}+\delta.
$$
Therefore, to prove Theorem~\ref{thmA} it is sufficient to prove that for any piecewise $\lambda$-affine map $f$ we have $\dim_H (\mathcal{E}_f\cap (-\delta_0,\delta_0))=0$ for some $\delta_0>0$.

Since $|\lambda|<1$, the map $f$ has a gap, i.e., $I\setminus f(I)\neq\emptyset$. This gap also has non-empty interior, so we can choose a point $c\in I\setminus f(I)$ such that $\ell:=\text{dist}(c,f(I))>0$. Here, $\text{dist}$ denotes the distance in $I$ induced from the circle $\Rr/\Zz$ through the canonical identification $I\hookrightarrow\Rr\to\Rr/\Zz$. 

Now, we define $g\colon I \to I $ by $$g=R_{-c}\circ f\circ R_{c}.$$ The map $g$ is again a piecewise $\lambda$-affine map of $I$. Clearly, $\ell=\text{dist}(0,g(I))$ since $R_c$ is an isometry of $(I,\text{dist})$. We claim that
\begin{equation}\label{claim1}
\dim_H (\mathcal{E}_g\cap (-\ell,\ell))=0.
\end{equation}
From \eqref{claim1} we conclude the proof of Theorem~\ref{thmA},  because $f$ and $g$ are conjugated through $R_c$, which implies that $\mathcal{E}_f= \mathcal{E}_g$. 

\begin{figure}
\centering
\begin{subfigure}{.5\textwidth}
  \centering
  \includegraphics[width=0.9\linewidth]{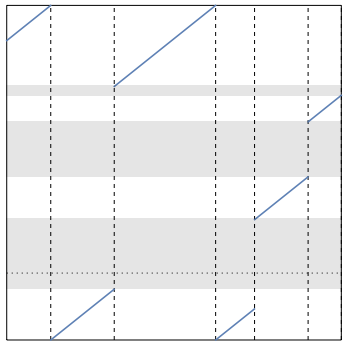}
  \put(-169,32){$c$}
  \caption{$f$}
\end{subfigure}%
\begin{subfigure}{.5\textwidth}
  \centering
  \includegraphics[width=0.9\linewidth]{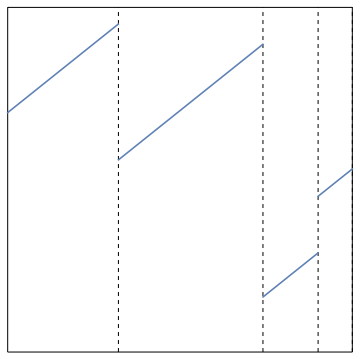}
  \caption{$g$}
\end{subfigure}
\caption{Plots of $f$ and $g=R_{-c}\circ f\circ R_c$.  The shaded region in (A) illustrates the gap of $f$.}
\label{fig2}
\end{figure}

In order to prove claim \eqref{claim1}, notice that, for every  $0\leq \delta<\ell$, we have $\text{dist}(0,g_\delta(I))>0$ where $g_\delta= R_\delta \circ g$. Equivalently, 
\begin{equation}\label{eq:property}
\overline{g_\delta(I)}\subset (0,1),\quad \forall\, \delta\in[0,\ell).
\end{equation}
Next,  extend $g_\delta$ in a canonical way  to a piecewise $\lambda$-affine map $G_\delta$ defined on the whole of $\Rr$. We call $G_\delta$ the affine extension of $g_\delta$. The map $G_\delta$ has the property that $G_\delta(x)=G(x)+\delta$ where $G\colon \Rr\to\Rr$ is the affine extension of $g$. Notice that this property holds because of \eqref{eq:property}. Moreover, since the orbit of any $x\in\Rr$ under $G_\delta$ eventually enters $I$, we conclude that $G_\delta$ is asymptotically periodic if and only if $g_\delta$ is asymptotically periodic. Thus, \eqref{claim1} is equivalent to
\begin{equation}\label{claim2}
\dim_H\left\{\delta\in(-\ell,\ell)\colon G_\delta \text{ is not asymptotically periodic}\right\}=0.
\end{equation}
 Let $(\Phi,\ba)$ be a pair defining $G$, i.e., an injective IFS $\Phi=\{\phi_1,\ldots,\phi_k\}$ and $\ba=(a_1,\ldots,a_{k-1})\in \Omega_k$ such that $G=f_{\Phi,\ba}$. Notice that the $\phi_i$'s are $\lambda$-affine maps. By Theorem~\ref{thmB}, we know that 
 $$
\dim_H\left\{ \delta\in  \Rr \colon f_{\Phi,\ba+\delta}\text{ is not asymptotically periodic}\right\}=0.
$$
It is easy to see that $G_\delta$ and $f_{\Phi,\ba -\delta/(1-\lambda)}$ are conjugated by the affine map $x\mapsto x+\delta/(1-\lambda)$ (cf. \cite[Reduction lemma]{NPR18}). This shows \eqref{claim2}, thus proving Theorem~\ref{thmA}. \qed
\section*{Acknowledgement}
The author was partially supported by the Project CEMAPRE/REM - UIDB/05069/2020 - financed by FCT/MCTES through national funds. The author also wishes to express its gratitude to Benito Pires and Arnaldo Nogueira for stimulating conversations.

\bibliographystyle{amsplain}

\end{document}